\documentclass[a4paper,twoside,10pt]{amsart}
\usepackage{amsmath,amsfonts,amssymb,amsthm}
\newtheorem{theorem}{Theorem}[section]
\newtheorem{proposition}{Proposition}[section]
\usepackage{mathrsfs}
\usepackage{color,graphicx}
\usepackage[a4paper]{geometry}
\numberwithin{equation}{section}

\author[G. Nemes]{Gerg\H{o} Nemes}
\address{School of Mathematics, The University of Edinburgh, James Clerk Maxwell Building, The King's Buildings, Peter Guthrie Tait Road, Edinburgh EH9 3FD, UK}
\email{Gergo.Nemes@ed.ac.uk}

\keywords{asymptotic expansions, error bounds, Hurwitz zeta function, polygamma functions, gamma function, Barnes $G$-function}
\subjclass[2010]{41A60, 30E15, 11M35, 33B15}

\begin{document}

\title[The asymptotics of the Hurwitz zeta function]{Error bounds for the asymptotic expansion\\ of the Hurwitz zeta function}

\begin{abstract} In this paper, we reconsider the large-$a$ asymptotic expansion of the Hurwitz zeta function $\zeta(s,a)$. New representations for the remainder term of the asymptotic expansion are found and used to obtain sharp and realistic error bounds. Applications to the asymptotic expansions of the polygamma functions, the gamma function, the Barnes $G$-function and the $s$-derivative of the Hurwitz zeta function $\zeta(s,a)$ are provided. A detailed discussion on the sharpness of our error bounds is also given.
\end{abstract}
\maketitle

\section{Introduction and main results}\label{sec1} In the present paper, we reconsider the large-$a$ asymptotic expansion of the Hurwitz zeta function which is defined for complex $s$ and $a$ by the series expansion
\[
\zeta(s,a) = \sum\limits_{n = 0}^\infty \frac{1}{\left( n + a \right)^s } ,
\]
when $\Re(s)>1$ and $a\notin \mathbb{Z}_{\leq 0}$. The Hurwitz zeta function has a meromorphic continuation in the $s$-plane with a simple pole of residue $1$ at $s=1$. The Riemann zeta function is a special case since $\zeta( s,\frac{1}{2}) =(2^s-1)\zeta(s)$ and $\zeta( s,1) = \zeta(s)$. As a function of $a$, with $s$ ($\neq 1$) fixed, $\zeta( s,a)$ is analytic in the domain $\left| {\arg a} \right| < \pi$ and (in general) possesses branch-point singularities at non-positive integer values of $a$.

It is well known that, as $a\to \infty$ in the sector $|\arg a|\leq \pi-\delta$, with $s$ ($\neq 1$) and $\delta>0$ being fixed, the Hurwitz zeta function has the asymptotic expansion
\begin{equation}\label{eq3}
\zeta(s,a) \sim \frac{1}{2}a^{ - s}  + \frac{a^{1 - s} }{s - 1} + a^{1 - s}\sum\limits_{n = 1}^{\infty} \frac{B_{2n}}{\left( {2n} \right)!}\frac{ ( s)_{2n-1}}{a^{2n} },
\end{equation}
where the $B_{2n}$'s denote the even-order Bernoulli numbers and $(w)_p = \Gamma(w + p)/\Gamma(w)$ is the so-called Pochhammer symbol (see, e.g., \cite[p. 25]{Magnus} or \cite[eq. 25.11.43]{NIST}). In the special case that $a=N$ is a large positive integer, \eqref{eq3} can be regarded as an asymptotic expansion for the tail of the Dirichlet series expansion of the Riemann zeta function $\zeta(s)$, because
\begin{equation}\label{eq10}
\zeta(s,N) = \sum\limits_{n = 0}^\infty  \frac{1}{\left( {n + N} \right)^s } = \sum\limits_{n = N}^\infty \frac{1}{n^s} = \zeta(s) - \sum\limits_{n = 1}^{N - 1} \frac{1}{n^s} .
\end{equation}
In this special case, formula \eqref{eq3} is most likely due to Euler (see, e.g., \cite[Chapter 6]{Euler}), who used it to compute the numerical values of $\zeta(s)$ for $s=2,3,\ldots,15,16$. Gram \cite{Gram} combined \eqref{eq3} and \eqref{eq10} to evaluate numerically the first several zeros of the Riemann zeta function $\zeta(s)$ in the critical strip $0<\Re(s)<1$. A refinement of Gram's analysis was provided by Backlund \cite{Backlund}, who also gave precise bounds for the remainder term of the asymptotic expansion \eqref{eq3} in the case that $a$ is a positive integer. Recently, Johannson \cite{Johansson} gave a numerical treatment of $\zeta(s,a)$ based on the asymptotic expansion \eqref{eq3}.

The main aim of the present paper is to derive new representations and bounds for the remainder of the asymptotic expansion \eqref{eq3}. Thus, for $| \arg a| < \pi$, $s \ne 1$ and any positive integer $N$, we define the $N$th remainder term $R_N(s,a)$ of the asymptotic expansion \eqref{eq3} via the equality
\begin{equation}\label{eq2}
\zeta(s,a) = \frac{1}{2}a^{ - s}  + \frac{a^{1 - s}}{s - 1} + a^{1 - s}\left( \sum\limits_{n = 1}^{N - 1} \frac{B_{2n}}{\left( {2n} \right)!}\frac{(s)_{2n-1}}{a^{2n}}  + R_N (s,a) \right).
\end{equation}
Throughout this paper, if not stated otherwise, empty sums are taken to be zero. The derivations of the estimates for $R_N (s,a)$ are based on new representations of this remainder term. Before stating the main results in detail, we introduce some notation. We denote
\[
\Pi_p(w)=\frac{{w^p }}{2}\left( {e^{\frac{\pi }{2}ip} e^{iw} \Gamma(1 - p,we^{\frac{\pi }{2}i} ) + e^{ - \frac{\pi }{2}ip} e^{ - iw} \Gamma ( 1 - p,we^{ - \frac{\pi }{2}i} )} \right),
\]
where $\Gamma(1 - p,w)$ is the incomplete gamma function. The function $\Pi_p(w)$ was originally introduced by Dingle \cite[pp. 407]{Dingle} and, following his convention, we refer to it as a basic terminant (but note that Dingle's notation slightly differs from ours, e.g., $\Pi_{p-1}(w)$ is used for our $\Pi_p(w)$). The basic terminant is a multivalued function of its argument $w$ and, when the argument is fixed, is an entire function of its order $p$. We shall also use the concept of the polylogarithm function $\operatorname{Li}_p(w)$ which is defined by the power series expansion
\[
\operatorname{Li}_p(w) = \sum\limits_{n = 1}^\infty \frac{w^n}{n^p}
\]
for $\left|w\right|<1$ and by analytic continuation elsewhere \cite[\S 25.12]{NIST}. The order $p$ of this function can take arbitrary complex values. Finally, we shall denote
\[
\chi \left( p \right) = \pi ^{\frac{1}{2}} \frac{\Gamma \left( \frac{p}{2} + 1 \right)}{\Gamma \left( \frac{p}{2} + \frac{1}{2} \right)},
\]
for any complex $p$ with $\Re(p) > 0$.

We are now in a position to formulate our main results. In Theorem \ref{thm1}, we give new representations for the remainder term $R_N(s,a)$.

\begin{theorem}\label{thm1} Let $N$ be a positive integer and let $s$ be an arbitrary complex number such that $s\neq 1$ and $\Re ( s ) > 1 - 2N$. Then
\begin{equation}\label{eq4}
R_N (s,a) = \left( { - 1} \right)^{N + 1} \frac{2}{{\Gamma (s)}}\frac{1}{{\left( {2\pi a} \right)^{2N} }}\int_0^{ + \infty } {\frac{{t^{s + 2N - 2} }}{{1 + (t/2\pi a)^2 }}\operatorname{Li}_{1 - s} (e^{ - t} )dt}
\end{equation}
provided $|\arg a|<\frac{\pi}{2}$, and
\begin{equation}\label{eq33}
R_N ( s,a) = \left( { - 1} \right)^{N + 1} 2\frac{ ( s)_{2N-1}}{\left( {2\pi a} \right)^{2N}}\sum\limits_{k = 1}^\infty \frac{\Pi _{s + 2N - 1} ( 2\pi ak)}{k^{2N}} ,
\end{equation}
\begin{equation}\label{eq88}
R_N (s,a) = \frac{{ (s)_{2N}}}{{\left( {2N} \right)!}} a^{s-1} \int_0^{ + \infty } {\frac{{B_{2N}  - B_{2N} ( t - \left\lfloor t \right\rfloor )}}{{\left( {t + a} \right)^{s + 2N} }}dt} 
\end{equation}
provided $|\arg a|<\pi$, where the $B_{2N}(t)$'s denote the even-order Bernoulli polynomials.
\end{theorem}

Our derivation of \eqref{eq33} is based on the integral representation \eqref{eq4}. Alternatively, formula \eqref{eq33} can be deduced from the more general expansion \eqref{eq25} of Paris or from an analogous result about the Lerch zeta function by Katsurada \cite[Theorem 1]{Katsurada}.

The subsequent theorem provides bounds for the remainder $R_N(s,a)$. The error bound \eqref{eq12} may be further simplified by employing the various estimates for $\mathop {\sup }\nolimits_{r \ge 1} \left| {\Pi_p ( 2\pi ar)} \right|$ given in Appendix \ref{appendixa}.

\begin{theorem}\label{thm2} Let $N$ be a positive integer and let $s$ be an arbitrary complex number such that $s\neq 1$ and $\Re ( s ) > 1 - 2N$. Then
\begin{equation}\label{eq12}
\left| R_N ( s,a) \right| \le \frac{{\left| {B_{2N} } \right|}}{{\left( {2N} \right)!}}\frac{{\left| {( s)_{2N-1}} \right|}}{{\left| a \right|^{2N} }}\mathop {\sup }\limits_{r \ge 1} \left| {\Pi _{s + 2N - 1} ( 2\pi ar)} \right|,
\end{equation}
\begin{equation}\label{eq111}
\left| R_N( s,a) \right| \le \frac{{\left| {B_{2N} } \right|}}{{\left( {2N} \right)!}}\frac{{\left| {(s)_{2N - 1}} \right|}}{{\left| a \right|^{2N} }}\frac{{\left| {s + 2N - 1} \right|}}{{\Re \left( s \right) + 2N - 1}}\sec ^{\Re \left( s \right) + 2N} \Big( {\frac{{\arg a}}{2}} \Big) \max ( 1,e^{-\Im \left( s \right)\arg a} ),
\end{equation}
provided that $|\arg a|<\pi$, and
\begin{equation}\label{eq113}
\left| R_N (s,a) \right| \le \frac{\left| B_{2N} \right|}{(2N)!}\frac{\left| {(s)_{2N - 1} } \right|}{\left| a \right|^{2N}}\left( 1 + \frac{\left| s + 2N - 1 \right|}{\Re (s) + 2N - 1}\chi (\Re (s) + 2N - 1)\max (1,e^{ - \Im (s)\arg a} ) \right),
\end{equation}
provided that $|\arg a|\leq \frac{\pi}{2}$.
\end{theorem}

In the special case that $a$ is a positive integer, the bound \eqref{eq111} is equivalent to that obtained by Backlund.

For the case of positive $a$ and real $s$, we shall show that the remainder term $R_N(s,a)$ does not exceed the first neglected term in absolute value and has the same sign provided that $s\neq 1$ and $s > 1 - 2N$. More precisely, we will prove that the following theorem holds.

\begin{theorem}\label{thm4} Let $N$ be a positive integer, $a$ be a positive real number, and let $s$ be an arbitrary real number such that $s\neq 1$ and $s > 1 - 2N$. Then
\begin{equation}\label{eq46}
R_N (s,a) = \frac{B_{2N} }{(2N)!}\frac{(s)_{2N - 1} }{a^{2N} }\theta _N (s,a),
\end{equation}
where $0<\theta _N (s,a)<1$ is a suitable number that depends on $s$, $a$ and $N$.
\end{theorem}

We remark that the special case of this theorem when $a$ is a positive integer is known and it appears, for example, in \cite[exer. 3.2, p. 292]{Olver}.

We would like to emphasize that the requirement $\Re(s) > 1 - 2N$ in the above theorems is not a serious restriction. Indeed, the index of the numerically least term of the asymptotic expansion \eqref{eq3} is $n \approx \pi |a|$. Therefore, it is reasonable to choose the optimal $N \approx \pi|a|$, whereas the condition $s = o(|a|)$ has to be fulfilled in order to obtain proper approximations from \eqref{eq2}.

Paris \cite{Paris} studied in detail the Stokes phenomenon associated with the asymptotic expansion of $\zeta(s,a)$ by establishing an exponentially improved extension of \eqref{eq3}. In particular, he proved that if $\left\{ {N_k } \right\}_{k = 1}^\infty$ is an arbitrary sequence of positive integers and $s$ is any complex number such that $s\neq 1$ and $\Re(s) > 1 - 2N_k$ for each $k\geq 1$, then the Hurwitz zeta function has the exact expansion
\begin{gather}\label{eq25}
\begin{split}
\zeta(s,a) = \frac{1}{2}a^{ - s}  + \frac{{a^{1 - s} }}{{s - 1}} + a^{1 - s}  & \left( \sum\limits_{k = 1}^\infty  \sum\limits_{n = 1}^{N_k  - 1} {\left( { - 1} \right)^{n + 1} 2\frac{{(s)_{2n-1}}}{{\left( {2\pi a k} \right)^{2n} }}}  \right. \\ & \left. + \sum\limits_{k = 1}^\infty  {\left( { - 1} \right)^{N_k  + 1} 2\frac{{(s)_{2N_k  - 1}}}{{\left( {2\pi a k} \right)^{2N_k } }}\Pi _{s + 2N_k  - 1} ( 2\pi a k)}  \right),
\end{split}
\end{gather}
as long as $|\arg a|<\pi$. Paris' proof of \eqref{eq25} uses techniques based on Mellin--Barnes integrals. We will give an alternative proof of his result in Section \ref{sec2}. Note that in his paper, Paris uses a notation of the basic terminant which is different from ours. Applying Stirling's formula and Proposition \ref{prop3}, it can readily be shown that if the $N_k$'s are chosen so that $N_k  = \pi|a|k + \mathcal{O}(1)$, as $a\to \infty$, then the infinite series involving the basic terminants converges very rapidly in the closed sector $|\arg a|\leq \frac{\pi}{2}$, with its $k$th term behaving like $(|a|k)^{\Re(s) - 1} e^{ - 2\pi |a|k}$.

The remaining part of the paper is structured as follows. In Section \ref{sec2}, we prove the representations for the remainder term stated in Theorem \ref{thm1} and the convergent expansion \eqref{eq25}. In Section \ref{sec3}, we prove the error bounds given in Theorems \ref{thm2} and \ref{thm4}. Section 4 provides applications of our results to the asymptotic expansions of the polygamma functions, the gamma function, the Barnes $G$-function and the $s$-derivative of the Hurwitz zeta function $\zeta(s,a)$. The paper concludes with a discussion in Section \ref{sec5}.

\section{Proof of the representations for the remainder term}\label{sec2}

In this section, we prove the representations for the remainder $R_N(s,a)$ stated in Theorem \ref{thm1} and the convergent expansion \eqref{eq25}.

We begin by deriving formula \eqref{eq4}. It is known that the remainder term $R_N(s,a)$ may be expressed in the form
\begin{equation}\label{eq5}
R_N(s,a) = \frac{a^{s-1}}{\Gamma(s)}\int_0^{ + \infty } {\left( {\frac{1}{{e^u  - 1}} - \frac{1}{u} + \frac{1}{2} - \sum\limits_{n = 1}^{N - 1} {\frac{{B_{2n} }}{{\left( {2n} \right)!}}u^{2n - 1} } } \right)u^{s - 1} e^{ - au} du} ,
\end{equation}
provided that $|\arg a|<\frac{\pi}{2}$, $s\neq 1$ and $\Re(s) > 1 - 2N$ (cf. \cite[p. 24]{Magnus} or \cite[eq. 25.11.28]{NIST}). Temme \cite[eq. (3.26), p. 64]{Temme} showed that
\begin{equation}\label{eq15}
\frac{1}{{e^u  - 1}} - \frac{1}{u} + \frac{1}{2} - \sum\limits_{n = 1}^{N - 1} {\frac{{B_{2n} }}{{\left( {2n} \right)!}}u^{2n - 1} }  =  \left( { - 1} \right)^{N+1} 2 \sum\limits_{k = 1}^\infty  {\frac{{u^{2N - 1} }}{{\left( {u^2  + 4 \pi ^2 k^2 } \right)\left( {2\pi k} \right)^{2N - 2} }}} 
\end{equation}
for any $u>0$ and positive integer $N$. Suppose for a moment that $a$ is real and positive. Substitution of the right-hand side of \eqref{eq15} into \eqref{eq5} yields
\begin{align}
R_N ( s,a) & =  \left( { - 1} \right)^{N+1} \frac{2 a^{s-1}}{\Gamma(s)} \int_0^{ + \infty } {\sum\limits_{k = 1}^\infty  {\frac{{u^{s+2N - 2} e^{ - au} }}{{\left( {u^2  + 4 \pi ^2 k^2} \right)\left( {2\pi k} \right)^{2N - 2} }}du} } \nonumber
\\ & =  \left( { - 1} \right)^{N+1} \frac{2 a^{s-1}}{\Gamma(s)} \sum\limits_{k = 1}^\infty  \int_0^{ + \infty } { \frac{{u^{s+2N - 2} e^{ - au} }}{{\left( {u^2  + 4 \pi ^2 k^2} \right)\left( {2\pi k} \right)^{2N - 2} }}du } \nonumber
\\ & = \left( { - 1} \right)^{N + 1} \frac{2}{{\Gamma (s)}}\frac{1}{{\left( {2\pi a} \right)^{2N} }}\sum\limits_{k = 1}^\infty  {\int_0^{ + \infty } {\frac{{t^{s + 2N - 2} }}{{1 + (t/2\pi a)^2 }}\frac{{e^{ - kt} }}{{k^{1 - s} }}dt} } \label{eq22} \\
& = \left( { - 1} \right)^{N + 1} \frac{2}{{\Gamma (s)}}\frac{1}{{\left( {2\pi a} \right)^{2N} }}\int_0^{ + \infty } {\frac{{t^{s + 2N - 2} }}{{1 + (t/2\pi a)^2 }}\sum\limits_{k = 1}^\infty  {\frac{{e^{ - kt} }}{{k^{1 - s} }}} dt} \nonumber \\ 
& = \left( { - 1} \right)^{N + 1} \frac{2}{{\Gamma (s)}}\frac{1}{{\left( {2\pi a} \right)^{2N} }}\int_0^{ + \infty } {\frac{{t^{s + 2N - 2} }}{{1 + (t/2\pi a)^2 }}\operatorname{Li}_{1 - s} (e^{ - t} )dt} . \nonumber
\end{align}
In passing to the third equality, we have performed the change of integration variable from $u$ to $t$ by $t = a u/ k$. The changes in the orders of summation and integration are permitted because of absolute convergence. Now an analytic continuation argument shows that this result is valid for any complex $a$ satisfying $|\arg a|<\frac{\pi}{2}$, and the proof of the representation \eqref{eq4} is complete.

The representation \eqref{eq33} can be proved as follows. By making the change of variable from $t$ to $v$ via $v = kt$ in \eqref{eq22} and employing the integral representation \eqref{eq150} of the basic terminant, we find 
\begin{align}
R_N (s,a) & = \left( { - 1} \right)^{N + 1} \frac{2}{\Gamma(s)}\frac{1}{{(2\pi a)^{2N} }}\sum\limits_{k = 1}^\infty  {\frac{1}{{k^{2N} }}\int_0^{ + \infty } {\frac{{v^{s + 2N - 2} e^{ - v} }}{{1 + (v/2\pi ak)^2 }}dv} } \label{eq61} \\ &
= \left( { - 1} \right)^{N + 1} 2\frac{ ( s)_{2N-1}}{{\left( {2\pi a} \right)^{2N} }}\sum\limits_{k = 1}^\infty  {\frac{{\Pi _{s + 2N - 1} ( 2\pi ak)}}{{k^{2N} }}}. \label{eq60}
\end{align}
Since $\Pi _p (w) = \mathcal{O}(1)$ as $w\to \infty$ in the sector $|\arg w|<\pi$, by analytic continuation, the representation \eqref{eq60} is valid in a wider range than \eqref{eq61} (or, equivalently, \eqref{eq4}), namely in $|\arg a|<\pi$. Alternatively, \eqref{eq33} can be deduced from \eqref{eq25} by setting $N_k=N$ for each $k\geq 1$.

We now turn to the proof of formula \eqref{eq88}. Our starting point is the integral representation
\[
R_N(s,a) =  - \frac{(s)_{2N - 1}}{{\left( {2N-1} \right)!}} a^{s-1} \int_0^{ + \infty } {\frac{{B_{2N - 1} ( t - \left\lfloor t \right\rfloor)}}{{\left( {t + a} \right)^{s + 2N - 1} }}dt} ,
\]
which is valid when $a>0$, $s\neq 1$ and $\Re(s) > 2 - 2N$ (cf. \cite[eq. 25.11.7]{NIST}). Integrating once by parts, we obtain
\begin{align}
R_N (s,a) & = \frac{B_{2N} }{\left( 2N \right)!}\frac{(s)_{2N - 1}}{a^{2N}} - \frac{(s)_{2N}}{\left( 2N \right)!} a^{s-1} \int_0^{ + \infty } \frac{B_{2N} ( t - \left\lfloor t \right\rfloor )}{\left( {t + a} \right)^{s + 2N} }dt \label{eq112} \\ & = \frac{(s)_{2N}}{\left( 2N \right)!} a^{s-1} \int_0^{ + \infty } \frac{B_{2N}  - B_{2N} ( t - \left\lfloor t \right\rfloor )}{\left( {t + a} \right)^{s + 2N} }dt . \label{eq101}
\end{align}
An analytic continuation argument then allows one to conclude that \eqref{eq101} is valid under the more general conditions $|\arg a|<\pi$, $s\neq 1$ and $\Re (s) > 1 - 2N$. This finishes the proof of formula \eqref{eq88}.

We close this section by proving the convergent expansion \eqref{eq25}. For this purpose, suppose that $|\arg a|<\frac{\pi}{2}$, $s\neq 1$ and $\Re(s) > -1$. Under these assumptions, formula \eqref{eq61} applies and gives
\begin{equation}\label{eq45}
R_1 (s,a) = \frac{2}{\Gamma (s)}\frac{1}{(2\pi a)^2}\sum\limits_{k = 1}^\infty \frac{1}{k^2}\int_0^{ + \infty } \frac{v^s e^{ - v}}{1 + (v/2\pi ak)^2}dv .
\end{equation}
Now let $\left\{ {N_k } \right\}_{k = 1}^\infty$ be an arbitrary sequence of positive integers. We can expand the integrands in \eqref{eq45} by applying the geometric series:
\[
\frac{1}{1 + (v/2\pi ak)^2} = \sum\limits_{n = 1}^{N_k  - 1} {\left( { - 1} \right)^{n + 1} \left( \frac{v}{2\pi ak} \right)^{2n - 2} }  + \left( { - 1} \right)^{N_k  + 1} \frac{1}{1 + (v/2\pi ak)^2}\left( \frac{v}{2\pi ak} \right)^{2N_k  - 2} .
\]
Substituting these expressions into \eqref{eq45} and using the the integral representation \eqref{eq150} of the basic terminant, we deduce
\begin{equation}\label{eq81}
R_1 (s,a) = \sum\limits_{k = 1}^\infty  \sum\limits_{n = 1}^{N_k  - 1} \left( { - 1} \right)^{n + 1} 2\frac{(s)_{2n - 1}}{(2\pi ak)^{2n}} + \sum\limits_{k = 1}^\infty  \left( - 1 \right)^{N_k  + 1} 2\frac{(s)_{2N_k  - 1}}{(2\pi ak)^{2N_k }}\Pi _{s + 2N_k  - 1} (2\pi ak) .
\end{equation}
By appealing to analytic continuation, this expansion is valid under the more general assumptions $|\arg a|<\pi$, $s\neq 1$ and $\Re (s) > 1 - 2N_k$ for each $k\geq 1$. By the definition \eqref{eq2} of the remainder term $R_1 (s,a)$, the expansion \eqref{eq81} is seen to be equivalent to \eqref{eq25}.

\section{Proof of the error bounds}\label{sec3}

In this section, we prove the bounds for the remainder term $R_N (s,a)$ given in Theorems \ref{thm2} and \ref{thm4}.

We begin with the proof of the bound \eqref{eq12}. Suppose that $N$ is a positive integer and $s$ is an arbitrary complex number such that $s \neq 1$ and $\Re(s)>1-2N$. From \eqref{eq33} we infer that
\[
\left| R_N (s,a) \right| \le 2\frac{\left| (s)_{2N - 1} \right|}{(2\pi \left| a \right|)^{2N} }\sum\limits_{k = 1}^\infty \frac{\left| {\Pi _{s + 2N - 1} (2\pi ak)} \right|}{k^{2N}}  \le 2\frac{\left| {(s)_{2N - 1} } \right|}{(2\pi \left| a \right|)^{2N}}\sum\limits_{k = 1}^\infty  \frac{1}{k^{2N}} \mathop {\sup }\limits_{r \ge 1} \left| \Pi _{s + 2N - 1} (2\pi ar) \right|,
\]
which, after using the well-known identity
\begin{equation}\label{eq37}
B_{2N} = ( - 1)^{N + 1} \frac{2(2N)!}{(2\pi )^{2N}}\sum\limits_{k = 1}^\infty \frac{1}{k^{2N}},
\end{equation}
simplifies to the required result \eqref{eq12}.

We proceed by proving formula \eqref{eq46}. Note that $0 < \Pi _p (w) < 1$ whenever $w > 0$ and $p > 0$ (see Proposition \ref{prop1}). Therefore, employing the representation \eqref{eq33} and observing that the terms of the infinite sum in \eqref{eq33} are all positive, we can assert that
\[
R_N (s,a) = ( - 1)^{N + 1} 2\frac{(s)_{2N - 1}}{(2\pi a)^{2N}}\sum\limits_{k = 1}^\infty \frac{1}{k^{2N}} \theta _N (s,a),
\]
for some $0<\theta _N (s,a)<1$ depending on $s$, $a$ and $N$. The identity \eqref{eq37} then shows that this expression is equivalent to the required one in \eqref{eq46}.

The bound \eqref{eq111} can be proved as follows. It is easy to see that, for any positive real $t$ and complex $a$ with $|\arg a|<\pi$, the following two inequalities hold:
\begin{equation}\label{eq810}
\frac{1}{\big| \left( t + a \right)^{s + 2N}  \big|} = \frac{1}{\left| {t + a} \right|^{\Re \left( s \right) + 2N}}e^{\Im \left( s \right)\arg \left( {t + a} \right)}  \le \frac{1}{\left| {t + a} \right|^{\Re \left( s \right) + 2N}}\max (1,e^{\Im \left( s \right)\arg a} )
\end{equation}
and
\begin{equation}\label{eq844}
\left| t + a \right|^2  = t^2  + \left| a \right|^2  + 2t\left| a \right|\cos (\arg a) = (t + \left| a \right|)^2  + 4t\left| a \right|\sin ^2 \Big( {\frac{\arg a}{2}} \Big) \ge (t + \left| a \right|)^2 \cos ^2 \Big( {\frac{\arg a}{2}} \Big).
\end{equation}
Also, by the Fourier series of the Bernoulli polynomials \cite[eq. 24.8.1]{NIST}, we have
\begin{equation}\label{eq811}
\left( { - 1} \right)^{N + 1} (B_{2N}  - B_{2N} (t - \left\lfloor t \right\rfloor )) = \frac{{2(2N)!}}{{(2\pi )^{2N} }}\sum\limits_{k = 1}^\infty  {\frac{{1 - \cos (2\pi kt)}}{{k^{2N} }}}  \ge 0
\end{equation}
for any $t>0$. Consequently, if $N \geq 1$ is fixed, the function $B_{2N}  - B_{2N} (t - \left\lfloor t \right\rfloor )$ does not change sign. From \eqref{eq88}, using \eqref{eq810}--\eqref{eq811}, we can conclude that
\begin{align}
\left| {R_N (s,a)} \right| & \le \frac{{\left| {(s)_{2N} } \right|}}{(2N)!}\left| {a^{s - 1} } \right|\int_0^{ + \infty } {\frac{{\left| {B_{2N}  - B_{2N} (t - \left\lfloor t \right\rfloor )} \right|}}{{\left| {(t + a)^{s + 2N} } \right|}}dt} \nonumber
\\ & \le \frac{{\left| {(s)_{2N} } \right|}}{(2N)!}\left| {a^{s - 1} } \right|\left| {\int_0^{ + \infty } {\frac{{B_{2N}  - B_{2N} (t - \left\lfloor t \right\rfloor )}}{{\left| {t + a} \right|^{\Re \left( s \right) + 2N} }}dt} } \right|\max (1,e^{\Im (s)\arg a} ) \nonumber
\\ & \le \frac{{\left| {(s)_{2N} } \right|}}{(2N)!}\left| {a^{s - 1} } \right|\left| {\int_0^{ + \infty } {\frac{{B_{2N}  - B_{2N} (t - \left\lfloor t \right\rfloor )}}{{(t + \left| a \right|)^{\Re \left( s \right) + 2N} }}dt} } \right|\sec ^{\Re \left( s \right) + 2N} \Big( {\frac{{\arg a}}{2}} \Big)\max (1,e^{\Im (s)\arg a} ) \nonumber
\\ & = \frac{{\left| {(s)_{2N} } \right|}}{{(\Re (s))_{2N} }}\frac{{\left| {a^{s - 1} } \right|}}{{\left| a \right|^{\Re (s) - 1} }}\left| {R_N (\Re (s),\left| a \right|)} \right|\sec ^{\Re \left( s \right) + 2N} \Big( {\frac{{\arg a}}{2}} \Big) \max (1,e^{\Im (s)\arg a} ) \nonumber
\\ & = \frac{{\left| {(s)_{2N} } \right|}}{{(\Re (s))_{2N} }}e^{ - \Im (s)\arg a} \left| {R_N (\Re (s),\left| a \right|)} \right|\sec ^{\Re \left( s \right) + 2N} \Big( {\frac{{\arg a}}{2}} \Big)\max (1,e^{\Im (s)\arg a} ) \nonumber
\\ & = \frac{{\left| {(s)_{2N} } \right|}}{{(\Re (s))_{2N} }}\left| {R_N (\Re (s),\left| a \right|)} \right|\sec ^{\Re \left( s \right) + 2N} \Big( {\frac{{\arg a}}{2}} \Big) \max (1,e^{ - \Im (s)\arg a} ). \label{eq55}
\end{align}
By Theorem \ref{thm4}, the factor $\left| {R_N (\Re (s),\left| a \right|)} \right|$ may be bounded as follows:
\[
\left| R_N (\Re (s),\left| a \right|) \right| \le \frac{\left| {B_{2N} } \right|}{(2N)!}\frac{(\Re (s))_{2N - 1} }{\left| a \right|^{2N} } = \frac{\left| {B_{2N} } \right|}{(2N)!}\frac{\left| (s)_{2N - 1} \right|}{\left| a \right|^{2N} }\frac{(\Re (s))_{2N}}{\left| {(s)_{2N} } \right|}\frac{\left| {s + 2N - 1} \right|}{\Re (s) + 2N - 1}.
\]
Substituting this inequality into \eqref{eq55} yields the prescribed estimate \eqref{eq111}.

We close this section by proving the estimate \eqref{eq113}. It follows from the first equality in \eqref{eq844} that for any positive real $t$ and complex $a$ with $|\arg a|\leq\frac{\pi}{2}$, the inequality $\left| t+a \right|^2  \ge t^2+\left| a \right|^2$ holds true. We also have $\left| B_{2N} (t - \left\lfloor t \right\rfloor ) \right| \le \left| B_{2N} \right|$ for any $t>0$ \cite[eq. 24.9.1]{NIST}. Applying these inequalities and \eqref{eq810} in \eqref{eq112}, we deduce that
\begin{align*}
\left| R_N (s,a) \right| & \le \frac{\left| {B_{2N} } \right|}{(2N)!}\frac{\left| (s)_{2N - 1}  \right|}{\left| a \right|^{2N}} + \frac{\left| {B_{2N} } \right|}{(2N)!}\left| (s)_{2N} \right|\left| a^{s - 1} \right|\int_0^{ + \infty } \frac{dt}{\left| {t + a} \right|^{\Re (s) + 2N}} \max (1,e^{\Im (s)\arg a} )
\\ & \le \frac{\left| B_{2N} \right|}{(2N)!}\frac{\left| {(s)_{2N - 1} } \right|}{\left| a \right|^{2N} } + \frac{\left| B_{2N}  \right|}{(2N)!}\left| (s)_{2N} \right|\left| a^{s - 1} \right|\int_0^{ + \infty } \frac{dt}{(t^2  + \left| a \right|^2 )^{\frac{\Re (s) + 2N}{2}}} \max (1,e^{\Im (s)\arg a} )
\\ & = \frac{\left| B_{2N} \right|}{(2N)!}\frac{\left| (s)_{2N - 1} \right|}{\left| a \right|^{2N}} + \frac{\left| {B_{2N} } \right|}{(2N)!}\frac{\left| {(s)_{2N} } \right|}{\left| a \right|^{2N}}\frac{\left| {a^s } \right|}{\left| a \right|^{\Re (s)}}\frac{1}{2}\int_0^{ + \infty } \frac{u^{ - \frac{1}{2}}}{(u + 1)^{\frac{\Re (s) + 2N}{2}} }du \max (1,e^{\Im (s)\arg a} )
\\ & = \frac{\left| B_{2N} \right|}{(2N)!}\frac{\left| (s)_{2N - 1} \right|}{\left| a \right|^{2N}} + \frac{\left| {B_{2N} } \right|}{(2N)!}\frac{\left| (s)_{2N}  \right|}{\left| a \right|^{2N}}e^{ - \Im (s)\arg a} \frac{\pi ^{\frac{1}{2}}}{2}\frac{\Gamma \big( \frac{\Re (s) + 2N - 1}{2} \big)}{\Gamma \big( \frac{\Re (s) + 2N}{2} \big)}\max (1,e^{\Im (s)\arg a} )
\\ & = \frac{\left| B_{2N} \right|}{(2N)!}\frac{\left| (s)_{2N - 1} \right|}{\left| a \right|^{2N}} + \frac{\left| B_{2N}  \right|}{(2N)!}\frac{\left| (s)_{2N - 1} \right|}{\left| a \right|^{2N}}\frac{\left| {s + 2N - 1} \right|}{\Re (s) + 2N - 1}\chi (\Re (s) + 2N - 1)\max (1,e^{ - \Im (s)\arg a} ),
\end{align*}
which is equivalent to \eqref{eq113}. In arriving at the final result, we have made the change of integration variable from $t$ to $u$ by $u = t^2 /\left| a \right|^2$ and used the known integral representation of the beta function \cite[eq. 5.12.3]{NIST} and the definition of $\chi(p)$.

\section{Application to related functions}\label{sec4}

In this section, we show how the remainder terms of the known asymptotic expansions of the polygamma functions, the gamma function, the Barnes $G$-function and the $s$-derivative of the Hurwitz zeta function $\zeta(s,a)$ may be expressed in terms of the remainder $R_N(s,a)$. Thus, the theorems in Section \ref{sec1} can be applied to obtain representations and bounds for the error terms of these asymptotic expansions. Some of the resulting representations and bounds are well-known in the literature but many of them, we believe, are new.

The asymptotic expansion of the digamma function $\psi(z)=\frac{d}{dz}\log \Gamma(z)$ may be written
\begin{equation}\label{eq107}
\psi(z) = \log z - \frac{1}{{2z}} - \sum\limits_{n = 1}^{N - 1} {\frac{B_{2n}}{2n}\frac{1}{z^{2n}}}  + R_N^{(\psi)} (z),
\end{equation}
where $N$ is any positive integer and $R_N^{(\psi)} (z)=\mathcal{O}_N(|z|^{-2N})$ as $z\to \infty$ in the sector $|\arg z|\leq \pi-\delta$, with $\delta>0$ being fixed (cf. \cite[p. 18]{Magnus} or \cite[eq. 5.11.2]{NIST}). (Throughout this section, we use subscripts in the $\mathcal{O}$ notations to indicate the dependence of the implied constant on certain parameters.) Employing the known limit \cite[p. 271]{WW}
\[
\psi(z) = \mathop {\lim }\limits_{s \to 1 + 0} \left( \frac{1}{s - 1} - \zeta(s,z) \right) 
\]
for the right-hand side of \eqref{eq2} and comparing the result with \eqref{eq107}, we obtain the following relation:
\begin{equation}\label{eq31}
R_N^{(\psi )} (z) =  - R_N (1,z).
\end{equation}
A combination of \eqref{eq31} and \eqref{eq111}, for example, reproduces the known bound for $R_N^{(\psi )} (z)$ (cf. \cite[\S 5.11(ii)]{NIST} or \cite[exer. 4.2, p. 295]{Olver}).

The polygamma function $\psi^{(k)}(z)=\frac{d^k}{dz^k}\psi(z)$, with $k$ being any fixed positive integer, has the asymptotic expansion
\[
\left( { - 1} \right)^{k + 1} \psi ^{(k)} (z) = \frac{\Gamma (k+1)}{2}\frac{1}{{z^{k + 1} }} + \sum\limits_{n = 0}^{N - 1} {\frac{{B_{2n} }}{(2n)!}\frac{{\Gamma ( k + 2n )}}{{z^{k+2n} }}}  + R_N^{(\psi ^{(k)})}(z),
\]
for any $N\geq 1$, where $R_N^{(\psi^{(k)})} (z)=\mathcal{O}_{N,k}(|z|^{-k-2N})$ as $z\to \infty$ in the sector $|\arg z|\leq \pi-\delta$, with any fixed $\delta>0$ (see, for instance, \cite[p. 18]{Magnus}). The functional relation \cite[eq. 25.11.12]{NIST}
\[
( - 1)^{k + 1} \psi ^{(k)} (z) = k!\zeta (k + 1,z)
\]
implies that
\[
R_N^{(\psi ^{(k)})}(z) = k!z^{ - k} R_N (k + 1,z).
\]

The standard asymptotic expansion of the logarithm of the gamma function takes the form
\begin{equation}\label{eq52}
\log \Gamma (z) = \left( {z - \frac{1}{2}} \right)\log z - z + \frac{1}{2}\log (2\pi ) + \sum\limits_{n = 1}^{N - 1} {\frac{{B_{2n} }}{{2n(2n - 1)}}\frac{1}{{z^{2n - 1} }}}  + R_N^{(\Gamma )} (z),
\end{equation}
where $N$ is any positive integer and $R_N^{(\Gamma )} (z)=\mathcal{O}_N(|z|^{-2N+1})$ as $z\to \infty$ in the sector $|\arg z|\leq \pi-\delta$, with $\delta>0$ being fixed (see, e.g., \cite[p. 12]{Magnus} or \cite[eq. 5.11.1]{NIST}). The asymptotic expansion \eqref{eq52} is usually attributed to Stirling however, it was first discovered by De Moivre (for a detailed historical account, see \cite[pp. 482--483]{Anders}). If we substitute \eqref{eq2} into the right-hand side of the equality \cite[eq. 25.11.18]{NIST}
\[
\log \Gamma (z) = \frac{1}{2}\log (2\pi ) + \left. {\frac{{\partial \zeta (s,z)}}{{\partial s}}} \right|_{s = 0} 
\]
and compare the result with \eqref{eq52}, we find that
\begin{equation}\label{eq51}
R_N^{(\Gamma )} (z) = \lim _{s \to 0} (zR_N (s,z)/s).
\end{equation}
A combination of \eqref{eq51} and \eqref{eq111}, for example, reproduces Stieltjes' error bound for the asymptotic expansion of $\log \Gamma (z)$ (cf. \cite[\S 5.11(ii)]{NIST} or \cite[eq. (4.06), p. 294]{Olver}). Other known estimates for $R_N^{(\Gamma)}(z)$ including those of Lindel\"{o}f \cite{Lindelof}, F. W. Sch\"{a}fke and A. Sattler \cite{Schafke}, and Guari\u{\i} \cite{Gurarii} can all be deduced as direct consequences of \eqref{eq51}, \eqref{eq12} and the bounds given in Appendix \ref{appendixa}. 

The $G$-function, introduced and studied originally by Barnes \cite{Barnes1}, has the asymptotic expansions
\begin{gather}\label{eq251}
\begin{split}
\log G(z+1) = \frac{1}{4}z^2  + z\log \Gamma (z + 1) & - \left( {\frac{1}{2}z^2  + \frac{1}{2}z + \frac{1}{{12}}} \right)\log z - \log A \\ & + \sum\limits_{n = 1}^{N - 1} {\frac{{B_{2n + 2} }}{{2n(2n + 1)(2n + 2)}}\frac{1}{z^{2n}}}  + R_{N}^{(G1)}(z)
\end{split}
\end{gather}
and
\begin{equation}\label{eq252}
\log G(z+1) =  - \frac{3}{4}z^2  + \frac{z}{2}\log (2\pi) + \left( {\frac{1}{2}z^2  - \frac{1}{{12}}} \right)\log z + \frac{1}{{12}} - \log A + \sum\limits_{n = 1}^{N - 1} {\frac{{B_{2n + 2} }}{{2n( 2n + 2)}}\frac{1}{{z^{2n} }}}  + R_{N}^{(G2)}(z),
\end{equation}
for any $N\geq 1$, where $R_{N}^{(G1)}(z)$ and $R_{N}^{(G2)}(z)$ are both $\mathcal{O}_N(|z|^{-2N})$ as $z\to \infty$ in the sector $|\arg z|\leq \pi-\delta$, with any fixed $\delta>0$. The asymptotic expansion \eqref{eq251} was established by Ferreira and L\'{o}pez \cite{Ferreira}, whereas \eqref{eq252} is due to Barnes \cite{Barnes2}. The Barnes $G$-function is related to the Hurwitz zeta function through the functional equation \cite[eq. (2.2)]{Nemes}
\begin{equation}\label{eq346}
\log G(z + 1) = z\log \Gamma (z + 1)  - z\log z + \frac{1}{{12}} - \log A - \left. {\frac{{\partial \zeta (s,z)}}{{\partial s}}} \right|_{s =  - 1} ,
\end{equation}
which, after using \eqref{eq2}, gives
\begin{equation}\label{eq347}
R_N^{(G1)} (z) = \mathop {\lim }\limits_{s \to  - 1} (z^2 R_{N + 1} (s,z)/s(s + 1)) .
\end{equation}
An application of the identity $\log \Gamma (z + 1) = \log \Gamma (z) + \log z$ and the expansion \eqref{eq52} for the right-hand side of \eqref{eq251} shows that
\[
R_N^{(G2)} (z) = R_N^{(G1)} (z) + zR_{N + 1}^{(\Gamma )} (z).
\]
Explicit expressions and bounds for the remainder term $R_N^{(G1)} (z)$ were given recently by the present author \cite{Nemes}. The results of the paper \cite{Nemes} are special cases or consequences of the relation \eqref{eq347} and Theorems \ref{thm1}--\ref{thm4}.

Elizalde \cite{Elizalde} studied the large-$a$ asymptotic behaviour of the $s$-derivative of the Hurwitz zeta function $\zeta(s,a)$ evaluated at negative integer values of $s$. In particular, he showed that
\[
\left. {\frac{{\partial \zeta (s,a)}}{{\partial s}}} \right|_{s =  - 1} = - \frac{1}{4}a^2  + \left( {\frac{1}{2}a^2  - \frac{1}{2}a + \frac{1}{{12}}} \right)\log a + \frac{1}{{12}} - \sum\limits_{n = 1}^{N - 1} {\frac{{B_{2n + 2} }}{{2n(2n + 1)(2n + 2)}}\frac{1}{{a^{2n} }}}  + R_N^{(\partial \zeta)} (-1,a)
\]
and
\begin{gather}\label{eq348}
\begin{split}
\left. {\frac{{\partial \zeta (s,a)}}{{\partial s}}} \right|_{s =  - 2} =  - \frac{1}{9}a^3  + \frac{1}{{12}}a & + \left( {\frac{1}{3}a^3  - \frac{1}{2}a^2  + \frac{1}{6}a} \right)\log a \\ &+ \sum\limits_{n = 1}^{N - 1} {\frac{{2B_{2n + 2} }}{{(2n - 1)2n(2n + 1)(2n + 2)}}\frac{1}{{a^{2n - 1} }}}  + R_N^{(\partial \zeta)} (-2,a),
\end{split}
\end{gather}
where $N$ is any positive integer and $R_N^{(\partial \zeta)} (-1,a)=\mathcal{O}_N(|a|^{-2N})$, $R_N^{(\partial \zeta)} (-2,a)=\mathcal{O}_N(|a|^{-2N+1})$ as $a\to \infty$ in the sector $|\arg a|\leq \frac{\pi}{2}-\delta$, with $\delta>0$ being fixed. From \eqref{eq346} and \eqref{eq347} we infer that
\[
R_N^{(\partial \zeta)} (-1,a) = -R_N^{(G1)} (a) = -\mathop {\lim }\limits_{s \to  - 1} (a^2 R_{N + 1} (s,a)/s(s + 1)) .
\]
Differentiating the right-hand side of \eqref{eq2} with respect to $s$ and comparing the result with \eqref{eq348}, we obtain
\[
R_N^{(\partial \zeta)} (-2,a) = \mathop {\lim }\limits_{s \to  - 2} (2a^3 R_{N + 1} (s,a)/s(s + 1)(s + 2)) ,
\]
provided $N\geq 2$. An immediate consequence of these relations is that Elizalde's expansions are actually valid in the wider sector $|\arg a|\leq \pi-\delta$, with any fixed $\delta>0$.

\section{Discussion}\label{sec5}

In this paper, we have derived new representations and estimates for the remainder term of the asymptotic expansion of the Hurwitz zeta function. In this section, we shall discuss the sharpness of our error bounds.

First, we consider the bound \eqref{eq12} with $s$ being real. In particular, the asymptotic expansions discussed in Section \ref{sec4} belong to this case. Let $N$ be any non-negative integer, $s$ a real number and $a$ a complex number. Suppose that $s \neq 1$, $s>1-2N$ and $|\arg a| < \pi$. Under these assumptions, it follows from Theorem \ref{thm2} and Propositions \ref{prop1}, \ref{prop3} and \ref{prop4} that
\begin{equation}\label{eq892}
\left| R_N (s,a) \right| \le \frac{\left| B_{2N} \right|}{(2N)!}\frac{\left| {(s)_{2N - 1} } \right|}{\left| a \right|^{2N}} \times \begin{cases} 1 & \text{ if } \; \left|\arg a\right| \leq \frac{\pi}{4}, \\ \min\Big(\left|\csc ( 2\arg a)\right|,1 + \cfrac{1}{2}\chi(s + 2N - 1)\Big) & \text{ if } \; \frac{\pi}{4} < \left|\arg a\right| \leq \frac{\pi}{2}, \\ \cfrac{\sqrt {2\pi (s + 2N - 1)} }{2\left| {\sin (\arg a)} \right|^{s + 2N - 1} } + 1 + \cfrac{1}{2}\chi (s + 2N - 1) & \text{ if } \; \frac{\pi}{2} < \left|\arg a\right| < \pi. \end{cases}
\end{equation}

By the definition of an asymptotic expansion, $\lim _{a \to \infty } \left| a \right|^{2N} \left| R_N (s,a) \right| = \frac{\left| B_{2N}  \right|}{(2N)!}\left| (s)_{2N - 1}  \right|$ for any fixed $N \geq 0$. Therefore, when $|\arg a| \leq \frac{\pi}{4}$, the estimate \eqref{eq892} and hence our error bound \eqref{eq12} cannot be improved in general.

Consider now the case when $\frac{\pi}{4} < \left|\arg a\right| \leq \frac{\pi}{2}$. The bound \eqref{eq892} is reasonably sharp as long as $\left|\csc ( 2\arg a)\right|$ is not very large, i.e., when $|\arg a|$ is bounded away from $\frac{\pi}{2}$. As $|\arg a|$ approaches $\frac{\pi}{2}$, the factor $\left|\csc ( 2\arg a)\right|$ grows indefinitely and therefore it has to be replaced by $1 + \frac{1}{2}\chi(s + 2N - 1)$. By Stirling's formula,
\begin{equation}\label{eq772}
\chi(s + 2N - 1) \sim \left( \frac{\pi (s + 2N - 1)}{2} \right)^{\frac{1}{2}}
\end{equation}
as $N\to +\infty$, and therefore the appearance of this factor in the bound \eqref{eq892} may give the impression that this estimate is unrealistic for large $N$. However, this is not the case, as the following argument shows. We assume that $s$ is fixed, $|s+2N-1 - |2\pi a||$ is bounded, $|\arg a| = \frac{\pi}{2}$ and $N$ is large. The saddle point method applied to \eqref{eq912} shows that when $|\arg w| = \frac{\pi}{2}$ and $|p - |w||$ is bounded, the asymptotics
\begin{equation}\label{eq771}
\left| \Pi _p (w) \right| \sim \frac{1}{2}\left( \frac{\pi p}{2} \right)^{\frac{1}{2}} ,\quad \left| \Pi _p (kw)\right| = o(\left| \Pi _p (w) \right|) \quad k \in \mathbb{Z}_{>1}
\end{equation}
hold as $p \to +\infty$. Combining the estimates \eqref{eq772}--\eqref{eq771} with the representation \eqref{eq33} yields
\begin{align*}
\left| {R_N (s,a)} \right| & \sim 2\frac{\left| {(s)_{2N - 1} } \right|}{\left| {2\pi a} \right|^{2N} }\left| \Pi _{s + 2N - 1} (2\pi a) \right| \sim \frac{\left| {B_{2N} } \right|}{(2N)!}\frac{\left| {(s)_{2N - 1} } \right|}{\left| a \right|^{2N} }\left| \Pi _{s + 2N - 1} (2\pi a) \right|
\\ & \sim \frac{\left| {B_{2N} } \right|}{(2N)!}\frac{\left| {(s)_{2N - 1} } \right|}{\left| a \right|^{2N}}\frac{1}{2}\left( \frac{\pi (s + 2N - 1)}{2} \right)^{\frac{1}{2}} \sim \frac{\left| {B_{2N} } \right|}{(2N)!}\frac{\left| {(s)_{2N - 1} } \right|}{\left| a \right|^{2N}} \left( 1 + \frac{1}{2}\chi(s + 2N - 1) \right) .
\end{align*}
Consequently, when $|\arg a|$ is equal or smaller but close to $\frac{\pi}{2}$, the estimate \eqref{eq892} and thus the error bound \eqref{eq12} cannot be improved in general.

Finally, assume that $\frac{\pi}{2}<|\arg a| <\pi$. Elementary analysis shows that the factor
\[
\frac{\sqrt {2\pi (s + 2N - 1)} }{2\left| {\sin (\arg a)} \right|^{s + 2N - 1} } + 1 + \frac{1}{2}\chi (s + 2N - 1),
\]
as a function of $N$, remains bounded, provided that $|\arg a| - \frac{\pi}{2} = \mathcal{O}(N^{-\frac{1}{2}})$. This gives a reasonable estimate for
the remainder term $R_N (s,a)$. Otherwise, $|\sin (\arg a)|^{s + 2N - 1}$ can take very small values when $N$ is large, making the bound \eqref{eq892} completely unrealistic in most of the sectors $\frac{\pi}{2}<|\arg a| <\pi$. This deficiency of the bound \eqref{eq892} (and \eqref{eq12}) is necessary and is due to the omission of certain exponentially small terms arising from the Stokes phenomenon related to the asymptotic expansion \eqref{eq3} of the Hurwitz zeta function (for a detailed discussion, see \cite{Paris}). Thus, the use of the asymptotic expansion \eqref{eq3} should be confined to the sector $|\arg a|\leq \frac{\pi}{2}$. For other ranges of $\arg a$, one should use the analytic continuation formulae
\begin{equation}\label{eq778}
R_N (s,a) = R_N (s,ae^{ \pm \pi i} ) + e^{ \pm \frac{\pi}{2}is} \frac{(2\pi )^s }{\Gamma (s)}a^{s - 1} \operatorname{Li}_{1 - s} (e^{ \mp 2\pi ia} ),
\end{equation}
which can be derived from the representation \eqref{eq4} using the residue theorem.

Let us now turn our attention to the case when $s$ is allowed to be complex in \eqref{eq12}. Let $N$ be a positive integer and let $s$ and $a$ be complex numbers such that $s\neq 1$, $\Re(s)>1-2N$ and $|\arg a| < \frac{\pi}{2}$. Under these assumptions, it follows from Theorem \ref{thm2} and Propositions \ref{prop1} that
\begin{equation}\label{eq414}
\left| R_N (s,a) \right| \le \frac{\left| B_{2N} \right|}{(2N)!}\frac{\left| (s)_{2N - 1}\right|}{\left| a \right|^{2N}}\frac{\Gamma (\Re (s) + 2N - 1)}{\left| \Gamma (s + 2N - 1) \right|} \times \begin{cases} 1 & \text{ if } \; \left|\arg a\right| \leq \frac{\pi}{4}, \\ \left|\csc ( 2\arg a)\right| & \text{ if } \; \frac{\pi}{4} < \left|\arg a\right| < \frac{\pi}{2}. \end{cases}
\end{equation}
Now, we make the assumptions that $\Re(s) + 2N - 1 \to +\infty$ and $\Im(s) = o(N^{\frac{1}{2}})$ as $N\to + \infty$. With these provisos, it can easily be shown, using for example Stirling's formula, that the quotient of gamma functions in \eqref{eq414} is asymptotically $1$ for large $N$. Consequently, if $N$ is large, the estimates \eqref{eq414} and \eqref{eq12} are reasonably sharp in most of the sector $|\arg a| <\frac{\pi}{2}$. If the assumption $\Im(s) = o(N^{\frac{1}{2}})$ is replaced by the weaker condition $\Im(s) = \mathcal{O}(N^{\frac{1}{2}})$, the quotient in \eqref{eq414} is still bounded and hence, the estimates \eqref{eq414} and \eqref{eq12} are still relatively sharp. Otherwise, if $\Re(s) + 2N - 1$ is small and $\Im(s)$ is much larger than $N^{\frac{1}{2}}$, this quotient may grow exponentially fast in $|\Im(s)|$, which can make the bound \eqref{eq414} completely unrealistic. Therefore, if the asymptotic expansion \eqref{eq3} is truncated just before its numerically least term, i.e., when $N \approx \pi |a|$, the estimate \eqref{eq414} is reasonable in most of the sector $|\arg a| <\frac{\pi}{2}$ as long as $s = \mathcal{O}(|a|^{\frac{1}{2}})$. Although \eqref{eq3} is valid in the sense of generalised asymptotic expansions \cite[\S 2.1(v)]{NIST} provided $s = o(|a|)$, the stronger assumption $s = \mathcal{O}(|a|^{\frac{1}{2}})$ guarantees the rapid decay of the first several terms of \eqref{eq3} for large $a$.

When $|\arg a|$ is equal or smaller but close to $\frac{\pi}{2}$, the estimate \eqref{eq414} may be replaced by
\begin{gather}\label{eq779}
\begin{split}
& \left| R_N (s,a) \right| \le \frac{\left| B_{2N} \right|}{(2N)!}\frac{\left| (s)_{2N - 1}\right|}{\left| a \right|^{2N}} \\
& \times \left( \frac{1}{2} + \frac{\left| s + 2N - 1 \right|}{2(\Re (s) + 2N - 1)}\chi (\Re (s) + 2N - 1)\max (1,e^{ - \Im (s)\arg a} ) + \frac{\Gamma (\Re (s) + 2N - 1)}{2\left| \Gamma (s + 2N - 1) \right|} \right),
\end{split}
\end{gather}
which can be derived from Theorem \ref{thm2} and Proposition \ref{prop3}. If $s$ satisfies the requirements posed in the previous paragraph and $\Im(s)\arg a$ is non-negative, the factor in the second line of \eqref{eq779} is $\mathcal{O}(N^{\frac{1}{2}})$ for large $N$. By an argument similar to that used in the case of real $s$, it follows that \eqref{eq779} is a realistic bound. Otherwise, if $\Re(s) + 2N - 1$ is small and $\Im(s)$ is much larger than $N^{\frac{1}{2}}$ or $\Im(s)\arg a$ is negative and large, the right-hand side of the inequality \eqref{eq779} may be a serious overestimate of the actual error.

Similarly to the case of real $s$, for values of $a$ outside the closed sector $|\arg a|\leq \frac{\pi}{2}$, one should make use of the continuation formulae \eqref{eq778}.

We continue by discussing the bound \eqref{eq113}. Let $N$ be a positive integer, $s$ a real number and $a$ a complex number. Suppose that $s\neq 1$, $s > 1 - 2N$ and $|\arg a|\le \frac{\pi}{2}$. Under these requirements, the bound \eqref{eq113} takes the form
\[
\left| R_N (s,a)\right| \le \frac{\left| B_{2N} \right|}{(2N)!}\frac{\left| (s)_{2N - 1} \right|}{\left| a \right|^{2N}}(1+\chi (s + 2N - 1)).
\]
It is easy to see that the right-hand side is always larger than the right-hand side of the inequality \eqref{eq892}. Hence, the estimate \eqref{eq892} (and thus the error bound \eqref{eq12}) is superior over \eqref{eq113}. However, for complex values of $s$ with $\Im(s) \neq 0$, the bound \eqref{eq113} can be sharper than \eqref{eq414} and \eqref{eq779} since it does not involve the gamma function ratio which grows exponentially fast in $|\Im(s)|$.

We conclude this section with a brief discussion of the error bound \eqref{eq111}. With the same assumptions as those for the estimate \eqref{eq892}, the bound \eqref{eq111} simplifies to
\[
\left| R_N (s,a) \right| \le \frac{\left| B_{2N}  \right|}{(2N)!}\frac{\left| (s)_{2N - 1} \right|}{\left| a \right|^{2N}}\sec^{s + 2N} \Big( \frac{\arg a}{2} \Big) .
\]
It is seen that, in general, this bound is weaker than that given in \eqref{eq892} and hence it is mainly of theoretical interest. In the case that $s$ is complex, the estimate \eqref{eq111} is reasonably sharp as long as $\Im(s)\arg a$ is non-negative and $|\arg a|$ is small.

\section*{Acknowledgement}

The author's research was supported by a research grant (GRANT11863412/70NANB15H221) from the National Institute of Standards and Technology. The author appreciates the help of Dorottya Szir\'aki in improving the presentation of some parts of the paper.

\appendix

\section{Bounds for the basic terminant}\label{appendixa}

In this appendix, we prove some estimates for the absolute value of the basic terminant $\Pi _p (w)$ with $\Re(p) > 0$. These estimates depend only on $p$ and the argument of $w$ and therefore also provide bounds for the quantity $\mathop {\sup }\nolimits_{r \ge 1} \left| {\Pi _{s + 2N - 1} ( 2\pi ar)} \right|$ which appears in Theorem \ref{thm2}.

\begin{proposition}\label{prop1}
For any complex $p$ with $\Re(p) > 0$ it holds that
\begin{equation}\label{eq65}
\left| {\Pi _p (w)} \right| \le \frac{\Gamma (\Re (p))}{\left| \Gamma (p) \right|} \times \begin{cases} 1 & \text{ if } \; \left|\arg w\right| \leq \frac{\pi}{4}, \\ \left|\csc ( 2\arg w)\right| & \text{ if } \; \frac{\pi}{4} < \left|\arg w\right| < \frac{\pi}{2}. \end{cases}
\end{equation}
Moreover, when $w$ and $p$ are positive, we have $0<\Pi _p (w)<1$.
\end{proposition}

We remark that it was shown by the author \cite{Nemes2} that $\left| {\Pi _p (w)} \right| \le \sqrt {\frac{e}{4}\left( {p + \frac{3}{2}} \right)}$ provided that $\frac{\pi }{4} < \left| {\arg w} \right| \le \frac{\pi }{2}$ and $p$ is real and positive, which improves on the bound \eqref{eq65} near $|\arg w|=\frac{\pi}{2}$.

\begin{proof} Our starting point is the integral representation \cite{Nemes2}
\begin{equation}\label{eq150}
\Pi _p (w) = \frac{1}{{\Gamma (p)}}\int_0^{ + \infty } {\frac{{t^{p - 1} e^{ - t} }}{{1 + (t/w)^2 }}dt} ,
\end{equation}
which is valid when $\left|\arg w\right|<\frac{\pi}{2}$ and $\Re(p)>0$. For $t \geq 0$, we have
\[
\bigg| {1 + \bigg(\frac{t}{w}\bigg)^2} \bigg| \ge \begin{cases} 1 & \text{ if } \; 0 \leq |\arg w| \leq \frac{\pi}{4}, \\ |\sin \left(2\arg w\right)| & \text{ if } \; \frac{\pi}{4} < |\arg w| < \frac{\pi}{2},\end{cases}
\]
and therefore
\[
\left| \Pi _p (w) \right| \le \frac{1}{\left| \Gamma (p) \right|}\int_0^{ + \infty } {\frac{{t^{\Re (p) - 1} e^{ - t} }}{{\left| {1 + (t/w)^2 } \right|}}dt}  \le \frac{\Gamma (\Re (p))}{\left| \Gamma (p) \right|} \times \begin{cases} 1 & \text{ if } \; \left|\arg w\right| \leq \frac{\pi}{4}, \\ \left|\csc ( 2\arg w)\right| & \text{ if } \; \frac{\pi}{4} < \left|\arg w\right| < \frac{\pi}{2}. \end{cases}
\]
In the case of positive $w$ and $p$, notice that $0 < 1/(1 + \left( {t/w} \right)^2)  < 1$ for any $t>0$. Therefore, the integral representation \eqref{eq150} combined with the mean value theorem of integration imply that $0<\Pi_p \left( w \right)<1$.
\end{proof}

\begin{proposition} For any complex $p$ with $\Re(p) > 0$, we have
\begin{equation}\label{eq441}
\left| \Pi_p (w) \right| \le \frac{1}{2}\sec ^{\Re(p)} (\arg w)\max (1,e^{\Im (p)\left(  \mp \frac{\pi }{2} - \arg w \right)} ) + \frac{1}{2}\max (1,e^{\Im (p)\left(  \pm \frac{\pi }{2} - \arg w \right)} )
\end{equation}
and
\begin{equation}\label{eq657}
\left| \Pi_p (w) \right| \le \frac{1}{2}\sec ^{\Re(p)} (\arg w)\max (1,e^{\Im (p)\left(  \mp \frac{\pi }{2} - \arg w \right)} ) + \frac{\Gamma (\Re (p))}{2\left| \Gamma (p) \right|},
\end{equation}
for $0 \le \pm\arg w < \frac{\pi }{2}$.
\end{proposition}

\begin{proof} It is enough to prove \eqref{eq441} when $0 \le \arg w < \frac{\pi }{2}$, because this inequality for the case $-\frac{\pi}{2} < \arg w \le 0$ can be derived by taking complex conjugates. To do so, we consider the integral representation
\begin{equation}\label{eq836}
\Pi _p (w) = \frac{1}{2}\int_0^{ + \infty } \frac{e^{ - t}}{(1 - it/w)^p }dt  + \frac{1}{2}\int_0^{ + \infty } \frac{e^{ - t}}{(1 + it/w)^p }dt ,
\end{equation}
which is valid when $|\arg w| < \frac{\pi}{2}$ and $\Re(p) > 0$. This integral representation follows from the definition of the basic terminant $\Pi _p (w)$ and a well-known representation of the incomplete gamma function \cite[eq. 8.6.5]{NIST}. To estimate the right-hand side of \eqref{eq836}, we note that for any positive real $t$ and complex $w$ with $0 \le \arg w < \frac{\pi }{2}$, the following inequalities hold:
\[
\frac{1}{{\left| {(1 \pm it/w)^p } \right|}} = \frac{1}{{\left| {1 \pm it/w} \right|^{\Re (p)} }}e^{\Im (p)(\arg (t \mp iw) \pm \frac{\pi }{2} - \arg w)}  \le \frac{1}{{\left| {1 \pm it/w} \right|^{\Re (p)} }}\max (1,e^{\Im (p)\left( { \pm \frac{\pi }{2} - \arg w} \right)} )
\]
and
\[
\left| 1 - it/w \right| \ge \cos (\arg w),\quad \left| 1 + it/w \right| \ge 1.
\]
Consequently, we obtain
\begin{align*}
\left| {\Pi _p (w)} \right| \le \; & \frac{1}{2}\int_0^{ + \infty } {\frac{{e^{ - t} }}{{\left| {1 - it/w} \right|^{\Re (p)} }}dt} \max (1,e^{ \Im (p)\left( {-\frac{\pi }{2} - \arg w } \right)} ) \\ & + \frac{1}{2}\int_0^{ + \infty } {\frac{{e^{ - t} }}{{\left| {1 + it/w} \right|^{\Re (p)} }}dt} \max (1,e^{  \Im (p)\left( {\frac{\pi }{2}-\arg w } \right)} )
\\ \le \; &  \frac{1}{2}\sec ^{\Re (p)} (\arg w)\max (1,e^{ \Im (p)\left( {-\frac{\pi }{2}-\arg w} \right)} ) + \frac{1}{2}\max (1,e^{ \Im (p)\left( { \frac{\pi }{2}-\arg w} \right)} ) .
\end{align*}
The estimate \eqref{eq657} can be proved in an analogous way, starting from the representation
\begin{equation}\label{eq777}
\Pi _p (w) = \frac{1}{2}\int_0^{ + \infty } \frac{e^{ - t} }{(1 - it/w)^p }dt  + \frac{1}{2\Gamma (p)}\int_0^{ + \infty } \frac{t^{p - 1} e^{ - t} }{1 + it/w}dt ,
\end{equation}
which is valid when $|\arg w| < \frac{\pi}{2}$ and $\Re(p) > 0$. This integral representation may be deduced from the definition of the basic terminant $\Pi _p (w)$ and two different representations of the incomplete gamma function \cite[eqs. 8.6.4 and 8.6.5]{NIST}.
\end{proof}

The following estimate was proved by the author in \cite{Nemes2} and is valid for positive real values of the order $p$.

\begin{proposition}\label{prop2} For any $p>0$ and $w$ with $\frac{\pi}{4} < \left| {\arg w} \right| < \pi$, we have
\begin{equation}\label{eq188}
\left| {\Pi _p \left( w \right)} \right| \le \frac{{\left| {\csc \left( {2\left( {\arg w - \varphi } \right)} \right)} \right|}}{{\cos ^p \varphi }},
\end{equation}
where $\varphi$ is the unique solution of the implicit equation
\[
\left( {p + 2} \right)\cos \left( {2\arg w -3\varphi} \right) = \left( {p - 2} \right)\cos\left(2\arg w-\varphi\right)
\]
that satisfies $0 < \varphi  <  - \frac{\pi}{4} + \arg w$ if $\frac{\pi}{4} < \arg w  < \frac{\pi}{2}$, $ - \frac{\pi}{2}  + \arg w  < \varphi  < -\frac{\pi}{4}+\arg w$ if $\frac{\pi}{2}  \le \arg w  < \frac{3\pi}{4}$, $ - \frac{\pi}{2}  + \arg w  < \varphi  < \frac{\pi }{2}$ if $\frac{3\pi}{4}  \le \arg w  < \pi$, $\frac{\pi}{4} + \arg w < \varphi  <  0$ if $-\frac{\pi }{2} < \arg w  < -\frac{\pi}{4}$, $\frac{\pi}{4}  + \arg w  < \varphi  < \frac{\pi}{2}+\arg w$ if $-\frac{3\pi}{4}  < \arg w  \le -\frac{\pi}{2}$ and $ - \frac{\pi}{2}  < \varphi  < \frac{\pi }{2}+ \arg w$ if $-\pi < \arg w \le -\frac{3\pi}{4}$.
\end{proposition}

We remark that the value of $\varphi$ in this proposition is chosen so as to minimize the right-hand side of the inequality \eqref{eq188}.

\begin{proposition}\label{prop3} For any complex $p$ with $\Re(p)>0$, we have
\begin{gather}\label{eq834}
\begin{split}
\left| {\Pi _p (w)} \right| \le \; & \frac{1}{2}  + \frac{\left| p \right|}{2 {\Re (p)} }{}_2F_1 \left( {\frac{1}{2},\frac{{\Re (p)}}{2};\frac{{\Re (p)}}{2} + 1;\sin ^2 (\arg w)} \right)\max (1,e^{ - \Im (p)\arg w} ) \\ & + \frac{1}{2}\max (1,e^{ \Im (p)\left( { \pm\frac{\pi }{2} - \arg w} \right)} )
\\  \le \; & \frac{1}{2} + \frac{{\left| p \right|}}{{2\Re (p)}}\chi (\Re (p))\max (1,e^{ - \Im (p)\arg w} ) + \frac{1}{2}\max (1,e^{\Im (p)\left( { \pm \frac{\pi }{2} - \arg w} \right)} )
\end{split}
\end{gather}
and
\begin{gather}\label{eq839}
\begin{split}
\left| {\Pi _p (w)} \right| \le \; & \frac{1}{2}  + \frac{\left| p \right|}{2 {\Re (p)} }{}_2F_1 \left( {\frac{1}{2},\frac{{\Re (p)}}{2};\frac{{\Re (p)}}{2} + 1;\sin ^2 (\arg w)} \right)\max (1,e^{ - \Im (p)\arg w} ) \\ & + \frac{\Gamma (\Re (p))}{2\left| \Gamma (p) \right|}
\\  \le \; & \frac{1}{2} + \frac{{\left| p \right|}}{{2\Re (p)}}\chi (\Re (p))\max (1,e^{ - \Im (p)\arg w} ) + \frac{\Gamma (\Re (p))}{2\left| \Gamma (p) \right|},
\end{split}
\end{gather}
for $\frac{\pi }{4} < \pm \arg w \le \frac{\pi }{2}$.
\end{proposition}

\begin{proof} It is sufficient to prove \eqref{eq834} for $\frac{\pi }{4} < \arg w \le \frac{\pi }{2}$, as the estimates for $-\frac{\pi }{2} \le \arg w < -\frac{\pi }{4}$ can be derived by taking complex conjugates. Integrating once by parts in \eqref{eq836}, we obtain
\begin{equation}\label{eq833}
\Pi _p (w) = \frac{1}{2} - \frac{p}{2iw}\int_0^{ + \infty } \frac{e^{ - t}}{(1 - it/w)^{p + 1}}dt  + \frac{1}{2}\int_0^{ + \infty } \frac{e^{ - t}}{(1 + it/w)^p}dt .
\end{equation}
Assuming that $\frac{\pi }{4} < \arg w \le \frac{\pi }{2}$, we can deform the contour of integration in \eqref{eq833} by rotating it through a right angle. And therefore, by Cauchy's theorem and analytic continuation, we have
\begin{equation}\label{eq912}
\Pi _p (w) = \frac{1}{2} - \frac{p}{2w}\int_0^{ + \infty } \frac{e^{ - it}}{(1 + t/w)^{p + 1}}dt  + \frac{1}{2}\int_0^{ + \infty } \frac{e^{ - t}}{(1 + it/w)^p}dt
\end{equation}
when $\frac{\pi }{4} < \arg w \le \frac{\pi }{2}$. Hence, we may assert that
\begin{align*}
\left| {\Pi _p (w)} \right| \le \; & \frac{1}{2} + \frac{{\left| p \right|}}{{2\left| w \right|}}\int_0^{ + \infty } {\frac{{dt}}{{\left| {(1 + t/w)^{p + 1} } \right|}}}  + \frac{1}{2}\int_0^{ + \infty } {\frac{{e^{ - t} }}{{\left| {(1 + it/w)^p } \right|}}dt} 
\\ \le \; & \frac{1}{2} + \frac{{\left| p \right|}}{{2\left| w \right|}}\int_0^{ + \infty } {\frac{{dt}}{{\left| {1 + t/w} \right|^{\Re (p) + 1} }}} \max (1,e^{ - \Im (p)\arg w} ) \\ & + \frac{1}{2}\int_0^{ + \infty } {\frac{{e^{ - t} }}{{\left| {1 + it/w} \right|^{\Re (p)} }}dt} \max (1,e^{ \Im (p)\left( {  \frac{\pi }{2} - \arg w} \right)} )
\\ \le \; & \frac{1}{2} + \frac{{\left| p \right|}}{{2\left| w \right|}}\int_0^{ + \infty } {\frac{{dt}}{{(1 + 2t/\left| w \right|\cos (\arg w) + t^2 /\left| w \right|^2 )^{\frac{{\Re (p) + 1}}{2}} }}} \max (1,e^{ - \Im (p)\arg w} ) \\ & + \frac{1}{2}\max (1,e^{ \Im (p)\left( { \frac{\pi }{2} - \arg w} \right)} )
\\ = \; & \frac{1}{2} + \frac{{\left| p \right|}}{2}\int_0^{ + \infty } {\frac{{du}}{{(1 + 2u\cos (\arg w) + u^2 )^{\frac{{\Re (p) + 1}}{2}} }}} \max (1,e^{ - \Im (p)\arg w} ) \\ & + \frac{1}{2}\max (1,e^{ \Im (p)\left( { \frac{\pi }{2} - \arg w} \right)} ).
\end{align*}
In the last step, we have performed the change of integration variable from $t$ to $u$ by $u = t/ |w|$. The $u$-integral can be evaluated in terms of the hypergeometric function (see \cite{Nemes2}) giving the first estimate in \eqref{eq834}. To obtain the second inequality in \eqref{eq834}, note that
\begin{align*}
{}_2F_1 \left( {\frac{1}{2},\frac{{\Re (p)}}{2};\frac{{\Re (p)}}{2} + 1;\sin ^2 (\arg w)} \right) & \le {}_2F_1 \left( {\frac{1}{2},\frac{{\Re (p)}}{2};\frac{{\Re (p)}}{2} + 1;1} \right)
\\ & = \frac{{\Gamma \left( {\frac{1}{2}} \right)\Gamma \big( {\frac{{\Re (p)}}{2} + 1} \big)}}{{\Gamma (1)\Gamma \big( {\frac{{\Re (p)}}{2} + \frac{1}{2}} \big)}} = \pi ^{\frac{1}{2}} \frac{{\Gamma \big( {\frac{{\Re (p)}}{2} + 1} \big)}}{{\Gamma \big( {\frac{{\Re (p)}}{2} + \frac{1}{2}} \big)}} = \chi (\Re (p)) .
\end{align*}
The bounds \eqref{eq839} can be proved in a similar way, starting from the representation \eqref{eq777}.
\end{proof}

\begin{proposition}\label{prop4} For any complex $p$ with $\Re(p)>0$, we have
\begin{gather}\label{eq930}
\begin{split}
\left| {\Pi _p (w)} \right| & \le e^{\Im (p)\left( { \pm \frac{\pi }{2} - \arg w} \right)} \frac{{\Gamma (\Re (p))}}{{\left| {\Gamma (p)} \right|}}\frac{{\sqrt {2\pi \Re (p)} }}{{2\left| {\sin (\arg w)} \right|^{\Re (p)} }} + |\Pi _p (we^{ \mp \pi i} )| \\ & \le e^{\Im (p)\left( { \pm \frac{\pi }{2} - \arg w} \right)} \frac{{\Gamma (\Re (p))}}{{\left| {\Gamma (p)} \right|}}\frac{{\chi (\Re (p))}}{{\left| {\sin (\arg w)} \right|^{\Re (p)} }} + |\Pi _p (we^{ \mp \pi i} )|,
\end{split}
\end{gather}
for $\frac{\pi}{2}<\pm \arg w <\pi$.
\end{proposition}

The dependence on $|w|$ in these estimates may be eliminated by employing the bounds for $|\Pi _p (we^{ \mp \pi i} )|$ that were derived previously.

\begin{proof} The proof is based on the functional relation \cite{Nemes2}
\begin{equation}\label{eq222}
\Pi _p (w) = \pm \pi i\frac{{e^{ \mp \frac{\pi }{2}ip} }}{{\Gamma (p)}}w^p e^{ \pm iw}  + \Pi _p (we^{ \mp \pi i} ).
\end{equation}
We take the upper or lower sign in \eqref{eq222} according as $\frac{\pi}{2}<\arg w < \pi$ or $-\pi <\arg w < -\frac{\pi}{2}$. Now, from \eqref{eq222} we can infer that
\[
\left| {\Pi _p (w)} \right| \le \pi e^{\Im (p)\left( { \pm \frac{\pi }{2} - \arg w} \right)} \frac{1}{{\left| {\Gamma (p)} \right|}}\left| w \right|^{\Re (p)} e^{ - \left| w \right|\left| {\sin (\arg w)} \right|}  + |\Pi _p (we^{ \mp \pi i} )|.
\]
Notice that the quantity $r^q e^{-r\alpha}$, as a function of $r > 0$, takes its maximum value at $r=q/\alpha$ when $\alpha > 0$ and $q > 0$. We therefore find that
\begin{align*}
\left| {\Pi _p (w)} \right| & \le \pi e^{\Im (p)\left( { \pm \frac{\pi }{2} - \arg w} \right)} \frac{1}{{\left| {\Gamma (p)} \right|}}\frac{{\Re (p)^{\Re (p)} e^{ - \Re (p)} }}{{\left| {\sin (\arg w)} \right|^{\Re (p)} }} + |\Pi _p (we^{ \mp \pi i} )|
\\ & \le e^{\Im (p)\left( { \pm \frac{\pi }{2} - \arg w} \right)} \frac{{\Gamma (\Re (p))}}{{\left| {\Gamma (p)} \right|}}\frac{{\sqrt {2\pi \Re (p)} }}{{2\left| {\sin (\arg w)} \right|^{\Re (p)} }} + |\Pi _p (we^{ \mp \pi i} )|.
\end{align*}
The second inequality can be obtained from the well-known fact that $\sqrt{2\pi}q^{q-\frac{1}{2}}e^{-q}\leq \Gamma(q)$ for any $q>0$ (see, for instance, \cite[eq. 5.6.1]{NIST}). We derive the second bound in \eqref{eq930} from the result that
\[
\left(\frac{q}{2}\right)^{\frac{1}{2}}  \le \frac{{\Gamma \left( {\frac{q}{2} + 1} \right)}}{{\Gamma \left( {\frac{q}{2} + \frac{1}{2}} \right)}}
\]
for any $q>0$ (see, e.g., \cite[eq. 5.6.4]{NIST}) and the definition of $\chi\left(q\right)$.
\end{proof}

Finally, we mention the following two-sided inequality proved by Watson \cite{Watson} for positive real values of $p$:
\[
\sqrt {\frac{\pi }{2}\Big( p + \frac{1}{2} \Big)}  < \chi(p) < \sqrt {\frac{\pi }{2}\Big( p + \frac{2}{\pi} \Big)} .
\]
The upper inequality can be used to simplify the error bounds involving $\chi(p)$.

\end{document}